\documentclass[12pt]{article}
\usepackage{latexsym}
\usepackage{amsmath}
\usepackage{amsfonts}
\usepackage{amsthm}
\usepackage{amssymb}

\def\0{\mathbf 0}
\def\1{\mathbf 1}

\numberwithin{equation}{section}
\newcounter{thm}[section]
\numberwithin{thm}{section}

\newtheorem{theorem}[thm]{Theorem}
\newtheorem{lemma}[thm]{Lemma}
\newtheorem{proposition}[thm]{Proposition}

\newtheorem{problem}{Problem}
\newtheorem{definition}[thm]{Definition}

 \title{Converence and submeasures in Boolean algebras}

 \author{Thomas Jech \\
    e-mail: jech@math.psu.edu
    }

\begin{document}

\maketitle

\centerline{In memory of Bohuslav Balcar}

\abstract A Boolean algebra carries a strictly positive exhaustive submeasure
if and only if it has a sequential topology that is uniformly Fr\'echet.
\endabstract

\section{Introduction.}

In 1937, John von Neumann proposed the problem of characterizing the complete Boolean algebras that carry a strictly positive
$\sigma-$additive measure (The Scottish Book \cite{Sc}, Problem 163.) Such algebras satisfy the {\em countable
chain condition} (ccc): every subset of mutually disjoint elements is countable. In addition, von Neumann
isolated a weak distributive law and asked whether these two properties are sufficient for the existence
of measure.

The first major contribution to von Neumann's problem was given by Dorothy Maharam \cite{Mah}
who identified a part of the problem as a topological problem. She showed that a complete Boolean algebra
carries a strictly positive continuous submeasure if and only if the sequential topology given by
the algebraic convergence is metrizable. This gives a natural decomposition of the von Neumann
problem into two parts:

I. Does every complete weakly distributive ccc Boolean algebra carry a strictly positive continuous submeasure?

II. Does the existence of a strictly positive continuous submeasure imply the existence of a strictly positive
$\sigma-$additive measure?

Part II has become the well known and well-studied Control Measure Problem \cite{Kal} in Functional Analysis
and remained open until Talagrand's solution \cite{Tal} in 2008.

As for Part I, Maharam showed that a negative answer would follow from the existence of a Suslin algebra.
Its existence is equivalent to the negation of Suslin's Hypothesis \cite{Sus} from 1920 but  was not known
until 1967 (cf. \cite{Ten}, \cite{Jech}.) The existence of a Suslin algebra is not provable; it is only consistent
with the axioms of set theory \cite{ST}. The work of Balcar, Jech and Paz\'ak \cite{BJP} in 2003 established that
the positive answer Part I of Maharam's problem is also consistent.

The work on Maharam's problem I followed in Maharam's footsteps and was carried out mostly by Bohuslav
Balcar and his collaborators (cf. \cite{BGJ}, \cite{BFH}, \cite{BJP}, \cite{BJ}, \cite{BJ2}), with additional
contributions of Stevo Todorcevic \cite{Tod}, \cite{Tod2}. The metrizability assumption of Maharam was
successively weakened to ``ccc and Hausdorff'' \cite{BGJ}, ``ccc, weakly distributive and $G_\delta$''
\cite{BJP}, ``weakly distributive and $\sigma-$finite cc'' \cite{Tod2}, and ``uniformly weakly distributive''
\cite{BJ}.

In this paper we generalize the topological characterization to finitely additive measures on Boolean algebras,
when neither the algebra nor the measure are complete. As the Control Measure Problem has been solved
by Talagrand, and there exists an exhaustive submeasure that is not equivalent to a measure, the goal
is to characterize Boolean algebras that carry an exhaustive submeasure. We give a necessary and sufficient
condition in  Theorem \ref{Thm}.

Our investigations were inspired by the work of Bohuslav Balcar who advocated the use of sequential topology
and who introduced me to Maharam's problems. This paper is dedicated to his memory.

\section{Submeasures on Boolean algebras}

A Boolean algebra is a set $B$ with Boolean operations 
$a \vee b$ (join), $a \wedge b$ (meet), $-a$ (complement) and
$a \vartriangle b$ (symmetric difference), the partial order $a\leq b$,
and the smallest and greatest element, $\mathbf 0$ and $\mathbf 1$.
To avoid trivialities we assume that $B$ is infinite and atomless.
We let $B^+=B-\{\0\}.$

An \emph{antichain}
in $B$ is a set $A \subset B$ whose any two elements $a,b$
are \emph{disjoint} i.e. $a \wedge b = \mathbf 0$. 
$B$ satisfies the \emph{countable chain condition} (ccc)
if it has no uncountable antichains.

\begin{definition}\label{sub}
A \emph{submeasure}  on a Boolean
algebra $B$ is a real valued function $m$ on $B$ such
that

\begin{itemize}

\item [(i)] $m(\mathbf 0) = 0      \text{ and } m(\mathbf 1) = 1 $,
\item [(ii)] $\text{if } a\leq b \text{ then } m(a)\leq m(b)$,
\item [(iii)] $m(a \vee b) \leq m(a) + m(b) $.

\end{itemize}

A submeasure $m$ is {\em strictly positive} if

\begin{itemize}
\item [(iv)] $m(a)=0 \text{ only if } a=\0.$
\end{itemize}

$m$ is a {\em measure} if
\begin{itemize}

\item [(v)] $m(a \vee b) =m(a) + m(b) \text{ whenever }  a \text{ and } b
\text{ are disjoint.}$

\end{itemize}

\end{definition}

Following Talagrand \cite{Ta}, a submeasure $m$ is {\em exhaustive} if for every infinite
antichain $\{a_n\}_n$, $\lim_n m(a_n)=0$; $m$ is {\em uniformly exhaustive} if for every 
$\varepsilon>0$ there exists a $K$ such that every antichain has at most $K$ elements with 
$m(a)\ge\varepsilon.$

Note that some condition like ``exhaustive'' is needed to exclude trivial examples
of submeasures like $m(a)=1$ for all $a\neq\0$. Every measure is uniformly exhaustive,
and uniformly exhaustive implies exhaustive.

If $m$ a strictly positive exhaustive submeasure, let, for each $n\in \mathbf N,$
$$C_n = \{a\in B: m(a)\ge 1/n\}.$$
Then $C_1\subset C_2\subset...\subset C_n\subset...$ and $\bigcup_n C_n = B^+,$
and for each $n$, every antichain in $C_n$ is finite. If $m$ is uniformly exhaustive then
there exist $K_n, n\in \mathbf N,$ such that for each $n$, every antichain in $C_n$ has size 
at most $K_n.$ We say that $B$ is $\sigma-${\em finite cc}, resp. $\sigma-${\em bounded cc}.

In \cite{KR}, Kalton and Roberts proved that if $m$ is a  uniformly exhaustive submeasure
then it is equivalent to a measure. This reduced the Control Measure Problem to the
question whether ther exists an exhaustive submeasure that is not uniformly exhaustive.
Such a submeasure was constructed by Talagrand in \cite{Tal}.

Let $m$ be a strictly positive exhaustive submeasure. Letting
$$\rho(a,b) = m(a\vartriangle b),$$
we obtain a distance function that makes $B$ a metric space. The metric $\rho$ defines 
a topology on $B$ that is invariant under Boolean operations; in particular if $U$ is an open 
neighborhood of $\0$ then $a\vartriangle U$ is an open neighborhood of $a$.

We shall isolate a property of topological Boolean algebras that is necessary and sufficient 
for the existence of an exhaustive submeasure.

\section{Convergence ideals on Boolean algebras}

We consider an abstract theory of convergent sequences in topological Boolean algebras.
This is inspired by an earlier work of Balcar, e.g. \cite{BGJ} and \cite{BFH}.

\begin{definition}

Let $B$ be a Boolean algebra. A {\em convergence ideal} on $B$ is a set $I$ of
infinite sequences $\{a_n\}_n$ in $B$      with the following properties:

\begin{itemize}

\item [(i)]  $\bigwedge_n a_n = \0$, i.e. there is no $a>\0$ such that $a\leq a_n$ for all $n$.

\item [(ii)] If $s\in I$ and if $t$ is an infinite subsequence of $s$ then $t\in I$.

\item [(iii)] If $\{a_n\}_n \in I$ and $b_n\le a_n$ for all $n$ then $\{b_n\}_n \in I$.

\item[(iv)] If $\{a_n\}_n \in I$ and $\{b_n\}_n \in I$ then $\{a_n \vee b_n\}_n  \in I$.

\item [(v)] If $\{a_n\}_n$ is an infinite antichain then $\{a_n\}_n \in I$.
 
\end{itemize}

\end{definition}

Given a convergence ideal $I$ we write $\lim_n a_n =\0$ instead of $\{a_n\}_n \in I$
and say that $\{a_n\}$ converges to $\0.$ We extend this to convergence to any $a$ by
$$\lim a_n = a \text{ whenever } \lim (a_n \vartriangle a) = \0.$$
This notion of convergence  has the obvious properties like

-each $a_n$ converges to at most one limit,

-each constant sequence converges to the constant

-$\lim(a_n \vee b_n) = \lim a_n \vee \lim b_n,$
etc.

\medskip

Given a convergence ideal we define a topology on $B$ as follows: a set $A$ is {\em closed} 
if it contains the limits of all sequences in $A$. We call such a topology {\em sequential}.

\begin{definition}
A sequential topology on $B$ is {\em Fr\'echet} if for every $A,$ the closure of $A$ is the set of
all limits of sequences in $A.$
\end{definition}

It should be clear that a sequential topology is Fr\'echet if and only if it has this property:

For every $A$, if for every $k\in \mathbf N$, $\{a^k_n\}_n$ is a sequence in $A$ with $\lim_n a^k_n=b_k$,
and if $\lim_k b_k=c,$ then there exists  a sequence $\{c_k\}_k$ in $A$ such that 
$\lim_k c_k=c.$

\begin{lemma}
A sequential topology on $B$ is Fr\'echet if and only if the convergence ideal has the
{\em diagonal property}:

if $\lim_n x^k_n = \0$ for all $k$, then there exists a function $F:\mathbf N \to \mathbf N$
such that $\lim_k x^k_{F(k)} =\0.$
\end{lemma}

\begin{proof}
Assume that $I$ has the diagonal property and let $a^k_n, b_k$ and $c$ be such that
$\lim_n a^k_n =b_k$ and $\lim_k b_k =c.$ We shall find a function $F:\mathbf N \to \mathbf N$
such that $\lim_k a^k_{F(k)} =c.$

Let $x^k_n=a^k_n\vartriangle b_k;$ then $\lim_n x^k_n=\0$ for each $k,$ and so there exists 
a function $F$ such that $\lim_k x^k_{F(k)}=\0.$ Now because

$$a^k_{F(k)}\vartriangle c \leq (a^k_{F(k)}\vartriangle b_k)\vee (b_k\vartriangle c),$$
we have
$$\lim_k(a^k_{F(k)}\vartriangle c ) \le \lim_k  (a^k_{F(k)}\vartriangle b_k)
\vee \lim_k (b_k \vartriangle c)=\0$$
and so $\lim_k a^k_{F(k)} =c.$

For the converse, assume that the sequential topology is Fr\'echet and let $x^k_n$ be such 
that $\lim_n x^k_n=\0$ for each $k.$ Let $A=\{a_k\}_k$ be an infinite antichain, and let, for each $n,$
$$y^k_n = x^0_n\vee x^1_n\vee ... \vee x^k_n, \text{ and } z^k_n = y^k_n \vee a_k.$$

For each $k$ we have $\lim_n y^k_n=\0$ and so $\lim_n z^k_n = a_k.$

Let $Z=\{z^k_n: k,n\in \mathbf N\}.$ Each $a_k$ is in the closure of $Z,$ and so is
$\0$ since $\0=\lim_k a_k.$ Because the topology is Fr\'echet, there is a sequence in $Z$
with limit $\0;$ hence there is an infinite set $E\subset \mathbf N \times \mathbf N$ such that 
$\lim \{z^k_n: (k,n)\in E\} = \0.$ For each $k$ there are only finitely many $n$ with $(k,n)\in E;$
this is because $\lim_n z^k_n=a_k.$ Hence there is an infinite subset $M$ of $\mathbf N$
and a function $G$ on $M$ such that $\lim \{z^m_{G(m)}:m\in M\}=\0.$
Let $M=\{m(0)<m(1)<...<m(k)<... : k\in \mathbf N\}$ and let $F(k)=G(m(k)).$ Then
we have, for every $k,$
$$x^k_{F(k)} \leq y^k_{F(k)} \le y^{m(k)}_{F(k)}\le z^{m(k)}_{G(m(k))}$$
and so $\lim_k x^k_F(k)=\0.$

\end{proof}

\begin{definition}\label{unif}
The sequential topology given by a convergence ideal $I$ is {\em uniformly Fr\'echet}
if there exist choice functions $F_k$ acting on infinite sequences such that, whenever
$\{x^k_n\}_n$, $k\in \mathbf N$, are sequences converging to $\0$, then
$$\lim_k F_k(\{x^k_n\}_n)=\0$$
(the $k$th term is the term of the $k$th sequence chosen by the $k$th function.)
\end{definition}

This property is the uniform version of the diagonal property: the element chosen
from the $k$th sequence depends only on $k$ and the sequence, not on the
other sequences.

\begin{proposition}
Let $m$ be a strictly positive exhaustive submeasure on $B$, and let $I$ be the ideal
of all sequences $\{a_n\}_n$ in $B$ such that $\lim_n m(a_n)=0.$ Then $I$
is a convergence ideal and the topology is uniformly Fr\'echet.
\end{proposition}

\begin{proof}
For $\{a_n\}_n \in I,$ $F_k$ chooses some $a_n$ such that $m(a_n)<1/k.$
\end{proof}

\begin{theorem}\label{Thm}
A Boolean algebra $B$ carries a strictly positive exhaustive submeasure if and only
if it has a uniformly Fr\'echet convergence ideal.
\end{theorem}

We prove the theorem in the next Section.

\section{Fragmentations}\label{frag}

If $m$ is a measure (or an exhaustive submeasure) on a Boolean algebra $B$, let for each $n$
$$C_n = \{a: m(a)\ge 1/n\}.$$
Every antichain in every $C_n$ is finite (or of bounded size if $m$ is uniformly exhaustive) and
so (in modern terminology), $B$ is $\sigma-$finite cc, or $\sigma-$bounded cc. These properties
were considered in \cite{HT} and \cite{Ga}. More recently, approximations of $B$ by a chain
$\{C_n\}_n$ of sets with various properties have proved useful, see e.g. \cite {BJ}, p. 260.
In Balcar's terminology, such approximations are called fragmentations.

\begin{definition}

A \emph{fragmentation} of a Boolean algebra $B$ is a sequence of subsets
$C_1\subset C_2\subset...\subset C_n\subset ...$ such that
$\bigcup_n C_n =B^+$ and for every $n$, if $a\in C_n$ and $a\le b$ then $b\in C_n$.

A fragmentation is \emph{$\sigma$-finite cc} if for every $n$, every
antichain in $C_n$ is finite.

A fragmentation is \emph{$\sigma$-bounded cc} if for every $n$
there is a constant $K_n$ such that every antichain $A\subset C_n$
has size $\le K_n$.

A fragmentation is \emph{graded} if for every $n$, whenever
$a\cup b\in C_n$ then either $a\in C_{n+1}$ or $b\in C_{n+1}$.

\end{definition}

In this Section we prove that if $B$ has a uniformly Fr\'echet ideal $I$ then
$B$ has a graded $\sigma-$finite cc fragmentation.

Let $I$ be a uniformly Fr\'echet ideal and let $F_k$, $k\in \mathbf N$, be choice functions
from definition \ref{unif}. For each $n$, let
$$V_n =\{a: a\le F_n(\vec x) \text{ for some } \vec x\in I\}, U_n = V_1\cap...\cap V_n, C_n = B-U_n.$$

We claim that $\{C_n\}_n$ is a fragmentation: it is enough to show that each $a\ne \0$ is in some $C_n.$
If $a\notin C_n$ for all $n$, then for every $n$ there exists some $\vec x^n\in I$ such that $a\le F_n(\vec x^n).$
By the uniform diagonality, $\{F_n(\vec x^n)\}_n \in I$ and it follows that $a=\0.$

A similar argument shows

\begin{lemma}\label{tight}
If $\{a_n\}_n$ is such that $a_n\notin C_n$ for each $n$ then $\{a_n\}_n\in I.$
\end{lemma}

\begin{proof}
For each $n$ there exists some $\vec x^n\in I$ such that $a_n\le F_n(\vec x^n)$
Because $\{F_n(\vec x^n)\}_n \in I$, we have $\{a_n\}_n\in I.$
\end{proof}

\begin{lemma}\label{Gdelta}
For every $n$, no $\vec a\in I$ is a subset of $C_n$.
\end{lemma}

\begin{proof}
If some $\vec a\in I$ is a subset of $C_n$ then for some $k\le n$,
some infinite subsequence $\vec b$ of $\vec a$ is disjoint from $V_k.$
But $F_k(\vec b)\in V_k.$
\end{proof}

Note that the fragmentation is $\sigma-$finite cc, because every infinite antichain is in $I$.

\begin{lemma}
For every $n$ there exists a $k$ such that $U_k \vee U_k \subset U_n$. Hence
$\{C_n\}_n$ has a subfragmentation that is $\sigma-$finite cc and graded.
\end{lemma}

\begin{proof}
Let $n\in \mathbf N$ and assume that for each $k$ there exist $a_k$ and $b_k$ in $U_k$
such that $c_k=a_k \vee b_k \in C_n.$ By Lemma \ref{tight}, $\{a_k\}_k$ and
$\{b_k\}_k$ are in $I$ and so $\{c_k\}_k \in I$. But $\{c_k\}_k \subset C_n,$
a contradiction.
\end{proof}

\section{Construction of an exhaustive submeasure}

let $B$ be a Boolean algebra and assume that $B$ has a graded $\sigma-$finite cc
fragmentation $\{C_n\}_n$. We shall define a submeasure on $B$.

For each $n$ let $U_n = B-C_n$ and $U_0=B$; we have $U_0\supset U_1\supset
...\supset U_n\supset ...$ and $\bigcap_n U_n = \{\0\}$.
For $X, Y\subset B$, let $X\vee Y$ denote the set $\{x\cup y: x\in X, y\in Y\}$.
As the fragmentation is graded, we have
$$U_{n+1}\vee U_{n+1}\subset U_n$$
for all $n\ge 0$. Cosequently, if $n_1<...<n_k$ then
$$U_{n_1+1}\vee...\vee U_{n_k +1}\subset U_{n_1}.$$
This is proved by induction on $k$.

Let $D$ be the set of all $r=\sum^k_{i=1} 1/2^{n_i}$ where $0<n_1<...<n_k$.
For each $r\in D$ we let $V_r=U_{n_1}\vee...\vee U_{n_k}$ and $V_1=U_0=B$.
For each $a\in B$ define
$$m(a)=\inf\{r\in D\cup\{1\}: a\in V_r\}.$$
Using the above property of the $U_n$, it follows that $V_r\subset V_s$ if $r<s$, and
$V_r\vee V_s\subset V_{r+s}$ when $r+s\le 1.$ From this we have $m(a)\le m(b)$
if $a\subset b$, and $m(a\cup b)\le m(a)+m(b).$

The submeasure $m$ is strictly positive because if $a\ne \0$ then for some $n$,
$a\notin U_n,$ and hence $m(a)\ge 1/2^n.$

Finally, $m$ is exhaustive: If $A$ is an infinite antichain, then for each $n$ only finitely many
elements of $A$ belong to $C_n$ and hence only finitely many have $m(a)\ge 1/2^n.$

\section{Measures and submeasures}

Following Maharam's work \cite{Mah}, the early work on the von Neumann Problem and the 
Control Measure Problem isolated the following classes of Boolean algebras: 
complete Boolean algebras that carry a $\sigma-$additive measure (measure algebars), complete
Boolean algebras that carry a continuous submeasure (Maharam algebras),
Boolean algebras that carry a finitely additive measure and Boolean algebras that
carry exhaustive submeasure. (For a detailed analysis, see Fremlin's article \cite{Fr}
in the Handbook of Boolean Algebras \cite{Handbook}, and Balcar et al. \cite{BGJ}.)

The Control Measure Problem was reduced to the question whether every exhaustive
submeasure is equivalent to a finitely additive measure. The result of Kalton and 
Roberts \cite{KR} shows that every uniformly exhaustive submeasure is equivalent to
a finitely additive measure, and in 2007, Michel Talagrand answered the
question in \cite{Tal} by constructing an exhaustive measure (on a countable Boolean
algebra) that is not uniformly exhaustive.

Returning to the characterization of algebras with submeasure by convergence,
we have the following:

\begin{definition}
Let $I$ be a convergence ideal on a Boolean algebra $B.$ $B$ is {\em $I-$ concentrated}
if for every sequence $\{A_n\}_n$ of finite antichains with $|A_n|\ge n$ there exists
a sequence $\{a_n\}_n\in I$ such that $a_n\in A_n$ for all $n.$
\end{definition}

\begin{theorem}\label{finitely add}
 A Boolean algebra carries a strictly positive finitely additive measure if and only if it has a uniformly
 Fr\'echet convergence ideal $I$ and is $I-$concentrated.
\end{theorem}

\begin{proof}
Let $\{C_n\}_n$ be the fragmentation constructed in  Section \ref{frag}. We claim that each
 $C_k$ has a bound on the size of antichains Otherwise, let $A_n$ be antichains in $C_k$
 of size at least $n$, and let $\{a_n\}_n\in I$ be such that $a_n\in A_n.$ This
 contradicts Lemma \ref{Gdelta}.
 
 It follows that $B$ has a $\sigma-$bounded cc graded fragmentation. Using the Kalton-Roberts
 method, one can verify Kelley's condition \cite{Ke} for the existence of a strictly positive measure;
 for details, see \cite{Je2}.
\end{proof}

We wish to mention a subtle point here. In the proof of Theorem \ref{finitely add} we use
a graded $\sigma-$bounded cc fragmentation. It is not enough to have a $\sigma-$bounded
cc algebra with a graded fragmentation: the metric completion of 
Talagrand's example is a Maharam algebra that does not carry a measure but is
$\sigma-$bounded cc.

\section{Continuity and weak distributivity}

Let $B$ be a complete Boolean algebra. $B$ has a $\sigma-$additive measure if and only if
it has finitely additive measure and is weakly distributive. This is explicitly stated in Kelley's
\cite{Ke} but was already known to Pinsker, see \cite{KVP}. Moreover, the proof
of continuity is exactly the same when applied to submeasures, see Fremlin \cite{Fr},
p. 946:

\begin{theorem}
Let $B$ be a complete Boolean algebra.

(a) $B$ carries a strictly positive $\sigma-$additive measure if and only if it is weakly
distributive and carries a strictly positive finitely additive measure.

(b) $B$ carries a strictly positive continuous submeasure if and only if it is weakly distributive
and carries a strictly positive exhaustive sumbeasure.
\end{theorem}

We note that (for ccc algebras) weak distributivity is equivalent to the Fr\'echet property for the algebraic
convergence (and uniform weak distributivity is equivalent uniformly Fr\'echet.) Thus:

\begin{theorem} 
Let $B$ be a complete Boolean algebra.

(a) $B$ is a Maharam algebra if and only if it is uniformly weakly
distributive \cite{BJ}, \cite{BJ2}.

(b) $B$ is a measure algebra if and only if it is uniformly wekaly 
distributive and concentrated \cite{J}.
\end{theorem}

\section{Fr\'echet and uniformly Fr\'echet}

The Fr\'echet property does not (provably) imply uniform Fr\'echet. This is because 
a Suslin algebra (the algebra obtained from a Suslin tree) is weakly distributive
but not $\sigma-$finite cc. This counterexample is only consistent with ZFC
(\cite{Ten}, \cite{Jech}) and does not exist in some models of set theory
\cite{ST}. Moreover, by \cite{BJP} it is consistent that every weakly distributive ccc 
complete Boolean algebra is a Maharam algebra (and hence uniformly weakly
distributive). This shows that Fr\'echet could imply uniformly Fr\'echet
for the algebraic convergence but the proof does not seem to work
in general. Thus we end with the following open problem:

\begin{problem}
Is there (in ZFC) a Boolean algebra with a convergence ideal $I$ that is Fr\'echet 
but not uniformly Fr\'echet?
\end{problem}


\bibliographystyle{plain}

\end{document}